\documentclass[12pt]{amsart}%
\usepackage{amscd,amsmath,latexsym,amsthm,amsfonts,amssymb,graphicx,color,geometry,hyperref}
\hypersetup{colorlinks=true,linkcolor=red, anchorcolor=green, citecolor=cyan, urlcolor=red, filecolor=magenta, pdftoolbar=true}
\usepackage[utf8]{inputenc}
\usepackage{amsmath}
\usepackage{amsfonts}
\usepackage{amssymb}
\usepackage{graphicx}%
\providecommand{\U}[1]{\protect\rule{.1in}{.1in}}
\providecommand{\U}[1]{\protect\rule{.1in}{.1in}}

\newtheorem{theorem}{Theorem}[section]

\newtheorem{remark}[theorem]{Remark}

\numberwithin{equation}{section}

\begin{document}
\setcounter{page}{1}

\title[Spaceability of the set of bounded linear non--absolutely summing operators]{Spaceability of the set of bounded linear non-absolutely summing operators in Quasi-Banach sequence spaces}

\author[Daniel Tomaz]{Daniel Tomaz}

\address{Department of Mathematics, Federal University of Para\'{i}ba, 	58.051-900 - Jo\~{a}o Pessoa, Brazil}
\email{\textcolor[rgb]{0.00,0.00,0.84}{danieltomazmatufpb@gmail.com}}


\thanks{Daniel Tomaz is supported by Capes}

\begin{abstract}
In the short note we prove  that  for every $0<p<1$, there exists an infinite dimensional closed linear subspace of $\mathcal{L}\left(  \ell_{p};\ell_{p}\right) $ every nonzero element of which is non $(r,s)$-absolutely summing operator for the real numbers $r,s$ with $1\leq s\leq r<\infty$. This improve a result obtained in \cite{DanielT}.   \\
\end{abstract} \maketitle

\section{Introduction}
In the last decade many authors have been searching for large linear structures of mathematical objects enjoying certain special properties. These notions of lineability/spaceability has been investigated in several contexts, for instance, Functional Analysis, Measure Theory, Probability Theory, Set Theory, etc.
\\If $E$ is a vector space, a subset $A$ of $E$ is said to be \emph{lineable} if  $A\cup\left\{  0\right\}  $ contains a infinite dimensional linear subspace of $E$. Moreover, if $E$ is a topological vector space, a subset $A$ is said \emph{spaceable} if $A\cup\left\{  0\right\}  $ contains a closed infinite dimensional linear subspace of $E$. If $\alpha$ is a cardinal number, a subset $A$ of $E$ is called  $\alpha$-\emph{lineable} (\emph{spaceable}) if $A\cup\left\{  0\right\}  $ contains a (closed) $\alpha$-dimensional linear subspace of $E$. \\These definitions were introduced by Aron, Gurariy and Seoane-Sep\'{u}lveda in the classical references \cite{Aron} and \cite{Quarta}, considered as the founding pillars of the theory of lineability. See also, for instance,  the recent papers \cite{Bernal,Pellegrino1,Pellegrino2,BCFP_LAA,cariellojfa,VMS}. We refer also the recent monograph \cite{book}, where many examples can be found and techniques are developed in several different frameworks.

\subsection{Notation}
Let us now fix some notation. Let $E,F$  be Banach  or quasi-Banach spaces over the scalar field $\mathbb{K}$, which can be either $\mathbb{R}$ or $\mathbb{C}$. The space of absolutely $(r,s)$-summing linear operators from $E$ to $F$  will be represented by
$\prod\nolimits_{(r,s)}\left(  E;F\right)  $ and the space of bounded linear
operators from $E$ to $F$  will be denoted by $\mathcal{L}\left(E;F\right)  $.
\\Recall that an linear operator $T:E\rightarrow F$ is absolutely
$\left(  r,s\right)  $-summing if $\sum_{k}\left\Vert T\left(  x_{k}\right)
\right\Vert ^{r}<\infty$ whenever $\left(  x_{k}\right)  _{k=1}^{\infty}$ is a
sequence in $E$ such that $\sum_{k}\left\vert f\left(  x_{k}\right)
\right\vert ^{s}<\infty$ for each $f\in E^{\prime}$, where $E^{\prime}$ denote the topological dual of $E$. \\The basics of the linear theory of absolutely summing operators can be found in the classical book \cite{Diestel}.
If $E$ is a Banach or quasi-Banach space, we denote by  \\ $ \ell_{p}^{w}\left(  E\right)  =\left\{  \left(  x_{j}\right)  _{j=1}^{\infty
}\in E^{\mathbb{N}};\text{ }%
{\textstyle\sum\limits_{j=1}^{\infty}}
\left\vert \phi\left(  x_{j}\right)  \right\vert ^{p}<\infty,\text{ }%
\forall\phi\in E^{\prime}\right\}$ the space of weakly $p$-summable $E$-valued sequences and by $ \ell_{p}\left(  E\right)  =\left\{  \left(  x_{j}\right)  _{j=1}^{\infty}\in
E^{\mathbb{N}};\text{ }%
{\textstyle\sum\limits_{j=1}^{\infty}}
\left\Vert x_{j}\right\Vert ^{p}<\infty\right\}$ the space of absolutely  $p$-summable $E$-valued sequences. We will denote by $\mathfrak{c}$ the cardinality of the continuum.

If $0<p<1$, the sequences spaces $\ell_{p}$ are quasi-Banach spaces ($p$-Banach space) with quasi-norms given by 
\[
\left\Vert x\right\Vert_{\ell_{p}}=\left(  \sum\limits_{k=1}^{\infty}\left\vert
x_{k}\right\vert ^{p}\right)  ^{\frac{1}{p}}.
\] 
\\The behavior of quasi-Banach spaces or, more generally, metrizable complete
topological vector spaces, called $F$-spaces is sometimes quite different
from the behavior of Banach spaces. Besides, the search for closed infinite dimensional subspaces of quasi-Banach spaces is a quite delicate issue . Thus, it seems interesting to look for lineability and spaceability techniques that also cover the case of quasi-Banach spaces. For more details on quasi-Banach spaces we refer to \cite{book3}.
The aim of this paper is to prove the spaceability of the  set of bounded linear non-absolutely summing operators in quasi-Banach sequence spaces. To be more precise, let us to prove that  $\mathcal{L}\left(  \ell_{p};\ell_{p}\right)  \diagdown\bigcup\limits_{1\leq
	s\leq r<\infty}%
{\textstyle\prod\nolimits_{\left(  r,s\right)  }}
\left(  \ell_{p};\ell_{p}\right)  $ is $\mathfrak{c}$-spaceable for every
$0<p<1$, improving a result that  was proved in \cite{DanielT}.
\section{preliminaries}
In this section, we will consider  some common tools in the related results to the lineability/spaceability.
Let us split $\mathbb{N}$ into countably many infinite pairwise disjoint
subsets $\left(  \mathbb{N}_{k}\right)  _{k=1}^{\infty}$. For each integer
$k\in\mathbb{N},$ write%
\[
\mathbb{N}_{k}=\left\{  n_{1}^{\left(  k\right)  }<n_{2}^{\left(  k\right)
}<\cdots\right\}  .
\]
Define 
\[
\ell_{p}^{\left(  k\right)  }:=\left\{  x\in\ell_{p}:\text{ }x_{j}=0\text{ if
}j\notin\mathbb{N}_{k}\right\}  .
\]
On the other hand, since $\mathbb{N=}\left\{  n_{m}^{\left(  j\right)
}:j,m\in\mathbb{N}\right\}  ,$ consider the sequence of linear operators%
\[
i^{\left(  k\right)  }:\ell_{p}\longrightarrow\ell_{p}^{\left(  k\right)  }%
\]
given by%
\[
\left(  i^{\left(  k\right)  }\left(  x\right)  \right)  _{n_{m}^{\left(
		j\right)  }}=\left\{
\begin{array}
[c]{c}%
x_{m},\text{ if }j=k,\\
0,\text{ if }j\neq k
\end{array}
\right.
\]
for all $x=\left(  x_{m}\right)  _{m=1}^{\infty}\in\ell_{p}$. Note that%
\[
\left\Vert i^{\left(  k\right)  }\left(  x\right)  \right\Vert _{\ell
	_{p}^{\left(  k\right)  }}=\left\Vert i\left(  x\right)  \right\Vert
_{\ell_{p}}%
\]
where $i:\ell_{p}\longrightarrow\ell_{p}$ is the identity map. Now, for each
$k\in\mathbb{N}$, consider the sequence $\left(  u_{k}\right)  _{k=1}^{\infty
}$ in $\mathcal{L}\left(  \ell_{p};\ell_{p}\right)  $ defined of the form%
\[
u_{k}:\ell_{p}\overset{i^{\left(  k\right)  }}{\longrightarrow}\ell
_{p}^{\left(  k\right)  }\overset{j_{k}}{\longrightarrow}\ell_{p},
\]
with $j_{k}:\ell_{p}^{\left(  k\right)  }\longrightarrow\ell_{p}$ is the
inclusion operator.  Moreover, notice that
\begin{equation}
\left\Vert i^{\left(  k\right)  }\left(  x\right)  \right\Vert _{\ell
	_{p}^{\left(  k\right)  }}=\left\Vert i\left(  x\right)  \right\Vert \label{rr}
_{\ell_{p}}=\left\Vert u_{k}\left(  x\right)  \right\Vert _{\ell_{p}}
\end{equation}
for all $k.$

\begin{theorem}
	\label{t1} (\cite[Theorem 4]{Maddox}) Let $0<p<1$ and $1\leq s\leq r<\infty$.
	Then the identity map $i:\ell_{p}\longrightarrow\ell_{p}$ is non $\left(
	r,s\right)  $-absolutely summing.
\end{theorem}

\begin{remark}
It is straightforward consequence of \eqref{rr} and the previous theorem that for each $k\in\mathbb{N}$, the operator $u_{k}:\ell_{p}\longrightarrow\ell_{p}$ is non $\left(
r,s\right)  $-absolutely summing regardless of the real numbers $r,s$, with
$1\leq s\leq r<\infty$.
\end{remark}

\section{The main result}
\begin{theorem}
	$\mathcal{L}\left(  \ell_{p};\ell_{p}\right)  \diagdown\bigcup\limits_{1\leq
		s\leq r<\infty}%
	{\textstyle\prod\nolimits_{\left(  r,s\right)  }}
	\left(  \ell_{p};\ell_{p}\right)  $ is $\mathfrak{c}$-spaceable for every
	$0<p<1$.
\end{theorem}

\begin{proof}
	In fact, by Theorem \ref{t1} it follows that $\mathcal{L}\left(  \ell_{p};\ell_{p}\right)  \diagdown\bigcup\limits_{1\leq
		s\leq r<\infty}%
	{\textstyle\prod\nolimits_{\left(  r,s\right)  }}
	\left(  \ell_{p};\ell_{p}\right)  $ is non-empty. So, consider  the operator $T:\ell_{p}\longrightarrow\mathcal{L}\left(
	\ell_{p};\ell_{p}\right)  $ defined by
	\[
	T\left(  \left(  a_{i}\right)  _{i=1}^{\infty}\right)  =\sum\limits_{i=1}%
	^{\infty}a_{i}u_{i},
	\]
	with $u_{i}$ defined in the preliminaries.
	It follows from \cite[Lemma 2.1]{DanielT} that $T$ is \ well-defined, linear and injective.
	Moreover, using \cite[Theorem 3.1]{DanielT} we know that
	\[
	T\left(  \ell_{p}\diagdown\left\{  0\right\}  \right)  \subset\mathcal{L}%
	\left(  \ell_{p};\ell_{p}\right)  \diagdown\bigcup\limits_{1\leq s\leq
		r<\infty}%
	{\textstyle\prod\nolimits_{\left(  r,s\right)  }}
	\left(  \ell_{p};\ell_{p}\right)  .
	\]
	Therefore, $\overline{T\left(  \ell_{p}\right)  }$ is a closed
	infinite-dimensional subspace of $\mathcal{L}\left(  \ell_{p};\ell_{p}\right)
	$. We just have to show that%
	\[
	\overline{T\left(  \ell_{p}\right)  }\diagdown\left\{  0\right\}
	\subset\mathcal{L}\left(  \ell_{p};\ell_{p}\right)  \diagdown\bigcup
	\limits_{1\leq s\leq r<\infty}%
	{\textstyle\prod\nolimits_{\left(  r,s\right)  }}
	\left(  \ell_{p};\ell_{p}\right)  .
	\]
	Indeed, let $\Psi\in\overline{T\left(  \ell_{p}\right)  }\diagdown\left\{
	0\right\}  $. Then, there are sequences $\left(  a_{i}^{\left(  k\right)
	}\right)  _{i=1}^{\infty}\in\ell_{p}\diagdown\left\{  0\right\}  $ $\left(
	k\in\mathbb{N}\right)  $ such that
	\begin{equation}
	\Psi=\lim_{k\rightarrow\infty}T\left(  \left(  a_{i}^{\left(  k\right)
	}\right)  _{i=1}^{\infty}\right)  \text{in}\ \mathcal{L}\left( \label{ee} \ell_{p};\ell_{p}\right).
	\end{equation}
	Note that, for each $k\in\mathbb{N}$,%
	\[
	T\left(  \left(  a_{i}^{\left(  k\right)  }\right)  _{i=1}^{\infty}\right)
	=\sum\limits_{i=1}^{\infty}a_{i}^{\left(  k\right)  }u_{i}.
	\]
	Then, from  \eqref{ee} we have
	\[
	\Psi=\lim_{k\rightarrow\infty}\sum\limits_{i=1}^{\infty}a_{i}^{\left(
		k\right)  }u_{i}=\sum\limits_{i=1}^{\infty}\lim_{k\rightarrow\infty}%
	a_{i}^{\left(  k\right)  }u_{i}.
	\]
	In particular, for $x\in\ell_{p}$ arbitrary we get
	\[
	\Psi(x)=\lim_{k\rightarrow\infty}\sum\limits_{i=1}^{\infty}a_{i}^{\left(
		k\right)  }u_{i}(x)=\sum\limits_{i=1}^{\infty}\lim_{k\rightarrow\infty}%
	a_{i}^{\left(  k\right)  }u_{i}(x).
	\]
	Since convergence in $\ell_{p}$ implies coordinatewise convergence, it follows
	that
	\begin{equation}
	\lim_{k\rightarrow\infty}a_{i}^{\left(  k\right)  }=\alpha_{i}\ \text{for all}\ i. \label{tt}
	\end{equation}
	On the other hand, since each operator $u_{i}\in\mathcal{L}\left(  \ell
	_{p};\ell_{p}\right)  \diagdown\bigcup\limits_{1\leq s\leq r<\infty}%
	{\textstyle\prod\nolimits_{\left(  r,s\right)  }}
	\left(  \ell_{p};\ell_{p}\right)  $ (it suffices to use \eqref{rr}), for each $i\in\mathbb{N}$, there exist a
	sequence $\left(  x^{\left(  j\right)  }\right)  _{j=1}^{\infty}\in\ell
	_{s}^{w}\left(  \ell_{p}\right)  $ such that
	 $\left(  u_{i}(x^{\left(
		j\right)  }\right) ) _{j=1}^{\infty}\notin\ell_{r}\left(  \ell_{p}\right)  $,
	that is,
	\begin{equation}
	\sum\limits_{j=1}^{\infty}\left\vert \varphi\left(  x^{\left(  j\right)
	}\right)  \right\vert ^{s}<\infty\text{ and }\sum\limits_{j=1}^{\infty \label{e1}
}\left\Vert u_{i}\left(  x^{\left(  j\right)  }\right)  \right\Vert _{\ell
_{p}}^{r}=\infty\text{ },
\end{equation}
for each $\varphi\in\left(  \ell_{p}\right)  ^{\prime}=\ell_{\infty}$ because
$0<p<1$ (see \cite[Theorem 2.3]{book1}). So, using \eqref{tt} we get
\begin{align*}
\sum\limits_{j=1}^{\infty}\left\Vert \Psi\left(  x^{\left(  j\right)
}\right)  \right\Vert _{\ell_{p}}^{r}  & =\sum\limits_{j=1}^{\infty}\left\Vert
\lim_{k\rightarrow\infty}\sum\limits_{i=1}^{\infty}a_{i}^{\left(  k\right)
}u_{i}\left(  x^{\left(  j\right)  }\right)  \right\Vert _{\ell_{p}}^{r}\\
& =\sum\limits_{j=1}^{\infty}\left\Vert \sum\limits_{i=1}^{\infty}%
\lim_{k\rightarrow\infty}a_{i}^{\left(  k\right)  }u_{i}\left(  x^{\left(
	j\right)  }\right)  \right\Vert _{\ell_{p}}^{r}\\
& =\sum\limits_{j=1}^{\infty}\left\Vert \sum\limits_{i=1}^{\infty}\alpha
_{i}.u_{i}\left(  x^{\left(  j\right)  }\right)  \right\Vert _{\ell_{p}}^{r}.
\end{align*}
 Since $\left(  \alpha_{i}\right)  _{i}\neq0$ (it follows from the use of the  $p$-norm in \eqref{tt}), let $i_{0}$ be such
that $\alpha_{i_{0}}\neq0$. Since the supports of the operators $u_{i}$ are pairwise disjoint for all $i$, from \eqref{e1} we have
\begin{align*}
\sum\limits_{j=1}^{\infty}\left\Vert \Psi\left(  x^{\left(  j\right)
}\right)  \right\Vert _{\ell_{p}}^{r}  & =\sum\limits_{j=1}^{\infty}\left\Vert
\sum\limits_{i=1}^{\infty}\alpha_{i}.u_{i}\left(  x^{\left(  j\right)
}\right)  \right\Vert _{\ell_{p}}^{r}\\
& \geq\sum\limits_{j=1}^{\infty}\left\Vert \alpha_{i_{0}}.u_{i_{0}}\left(  x^{\left(
	j\right)  }\right)  \right\Vert _{\ell_{p}}^{r}\\
& =\left\vert \alpha_{i_{0}}\right\vert ^{r}.\sum\limits_{j=1}^{\infty}\left\Vert
u_{i_{0}}\left(  x^{\left(  j\right)  }\right)  \right\Vert _{\ell_{p}}^{r}=\infty
\end{align*}
and thus
\[
\sum\limits_{j=1}^{\infty}\left\Vert \Psi\left(  x^{\left(  j\right)
}\right)  \right\Vert _{\ell_{p}}^{r}=\infty.
\]
We conclude that $\Psi\in\mathcal{L}\left(  \ell_{p};\ell_{p}\right)
\diagdown\bigcup\limits_{1\leq s\leq r<\infty}%
{\textstyle\prod\nolimits_{\left(  r,s\right)  }}
\left(  \ell_{p};\ell_{p}\right)  $. Hence,%
\[
\overline{T\left(  \ell_{p}\right)  }\diagdown\left\{  0\right\}
\subset\mathcal{L}\left(  \ell_{p};\ell_{p}\right)  \diagdown\bigcup
\limits_{1\leq s\leq r<\infty}%
{\textstyle\prod\nolimits_{\left(  r,s\right)  }}
\left(  \ell_{p};\ell_{p}\right)
\]
finishing the proof.
\end{proof}

\bibliographystyle{amsplain}

\end{document}